\documentclass[secthm,seceqn,amsthm,ussrhead,12pt]{amsart}
\usepackage{amsmath,latexsym}
\usepackage[english]{babel}
\usepackage[psamsfonts]{amssymb}
\usepackage{times}
\usepackage{cite}

\usepackage[mathcal]{euscript}
\numberwithin{equation}{section} \textwidth=17.5cm
 \newtheorem{thm}{Theorem}[section]
 
 \newtheorem{lem}[thm]{Lemma}
 
 \theoremstyle{definition}
 
 \theoremstyle{remark}

 \numberwithin{equation}{section}

\numberwithin{equation}{section}

\begin{document}

\title[Derivations on the algebra of $\tau$-compact operators]{Spatiality
of derivations on the algebra of $\tau$-compact operators}

\author{Shavkat Ayupov and Karimbergen Kudaybergenov}

\address[Shavkat Ayupov]{Institute of
 Mathematics, National University of
Uzbekistan,
 100125  Tashkent,   Uzbekistan
 and
 the Abdus Salam International Centre
 for Theoretical Physics (ICTP),
  Trieste, Italy}
\email{sh$_{-}$ayupov@mail.ru}

\address[Karimbergen Kudaybergenov]{Department of Mathematics, Karakalpak state
university, Nukus 230113, Uzbekistan.} \email{karim2006@mail.ru}

\maketitle
\begin{abstract}
This paper is devoted to derivations on the algebra $S_0(M, \tau)$
of all $\tau$-compact operators affiliated with a von Neumann
algebra $M$ and a faithful normal semi-finite trace $\tau.$ The
main result asserts that every $t_\tau$-continuous derivation
$D:S_0(M, \tau)\rightarrow S_0(M, \tau)$ is spatial and
implemented by a $\tau$-measurable operator affiliated with $M$,
where $t_\tau$ denotes the measure topology  on  $S_0(M, \tau)$.
We also show the automatic $t_\tau$-continuity of all derivations
on $S_0(M, \tau)$ for properly infinite von Neumann algebras $M$.
Thus in the properly infinite case the condition of
$t_\tau$-continuity of the derivation  is redundant for its
spatiality.
\end{abstract} \maketitle

\bigskip

\section{Introduction}

\medskip

Given an algebra $\mathcal{A},$ a linear operator
$D:\mathcal{A}\rightarrow \mathcal{A}$ is called a
\textit{derivation}, if $D(xy)=D(x)y+xD(y)$ for all $x, y\in
\mathcal{A}$ (the Leibniz rule). Each element $a\in \mathcal{A}$
implements a derivation $D_a$ on $\mathcal{A}$ defined as
$D_a(x)=[a, x]=ax-xa,$ $x\in \mathcal{A}.$ Such derivations $D_a$
are said to be \textit{inner derivations}. If the element $a,$
implementing the derivation $D_a,$ belongs to a larger algebra
$\mathcal{B}$ containing $\mathcal{A},$ then $D_a$ is called
\textit{a spatial derivation} on $\mathcal{A}.$

One of the main problems in the theory of derivations is to prove
the automatic continuity, ``innerness'' or ``spatiality'' of
derivations, or to show the existence of non-inner, non-spatial and moreover
discontinuous derivations on various topological algebras.

 The theory of derivations in the framework of operator algebras is an
  important and well investigated part of this theory, with applications in mathematical physics.
 It is well known that every derivation of a
  $C^{\ast}$-algebra is bounded (i.e. is norm continuous), and that every derivation of a von
 Neumann algebra is  inner.
 For  a detailed exposition of the theory of bounded derivations we refer to the monographs of
 Sakai \cite{Sak1}, \cite{Sak2}.

 Investigations of general unbounded
 derivations (and derivations on unbounded operator algebras)
 began much later and were motivated mainly by needs of mathematical physics,
  in particular by the
 problem of constructing the dynamics in quantum statistical mechanics
 (see, e.g.  \cite{Bra}, \cite{Sak1}, \cite{Sak2}).

 The development of a non commutative integration theory was
  initiated by I.~Segal \cite{Seg},
 who considered new classes of (not
  necessarily Banach) algebras of unbounded operators, in particular the
 algebra  $S(M)$ of all measurable operators affiliated with a von
 Neumann algebra $M.$ Algebraic, order and topological properties of
 the algebra $S(M)$ are somewhat similar to those of von Neumann algebras, therefore in
 \cite{Ayu1} the first author have initiated the study of derivations on the algebra $S(M).$

  If the von Neumann algebra $M$ is abelian, then it is *-isomorphic to the algebra
   $L^{\infty}(\Omega)=L^{\infty}(\Omega, \Sigma,
\mu)$ of all (classes of equivalence of) complex essentially
bounded measurable functions on a measure space $(\Omega, \Sigma,
\mu)$ and therefore  $S(M)\cong L^{0}(\Omega),$   where
$L^{0}(\Omega)=L^{0}(\Omega, \Sigma, \mu)$ the algebra of all
complex measurable functions on $(\Omega, \Sigma, \mu).$ In  this
case it is clear that  every measure continuous derivation and,
in particular, all inner derivations on $S(M)$ are identically zero, i.e.
trivial.

Investigating the abelian case A.~F.~Ber, F.~A.~Sukochev,
V.~I.~Chilin in  \cite{Ber} obtained necessary and sufficient
conditions for the existence of non trivial derivations on
commutative regular algebras. In particular they have proved that
the algebra  $L^{0}(0, 1)$ of all complex measurable functions on
the $(0, 1)$-interval admits non trivial derivations.
Independently A.~G.~Kusraev \cite{Kus} by means of Boolean-valued
analysis has established necessary and sufficient conditions for
the existence of non trivial derivations and automorphisms on
extended complete complex $f$-algebras. In particular he has also
proved the existence of non trivial derivations and automorphisms
on $L^{0}(0, 1).$ It is clear that these derivations are
discontinuous in the measure topology, and hence they are
not inner.

Therefore the properties of derivations on the unbounded operator
algebra $S(M)$ turned to be very far from   those on
$C^{\ast}$- or von Neumann algebras. Nevertheless there was a
conjecture that the existence of such "exotic" examples of
derivations is deeply connected with the commutative nature of the
underlying von Neumann algebra $M.$ In view of this  we have
initiated investigations of the above problems in the non
commutative case, namely on the algebra $LS(M)$ of all locally
measurable operators affiliated with a von Neumann algebra $M$,
and on its  subalgebras, including  the mentioned algebra $S(M)$, the
algebra $S(M,\tau)$ of all  $\tau$-measurable operators affiliated
with $M$, non commutative Arens algebras, etc. \cite{AAK2007,
AK2007, AAK2008, AAK2009, Alb2, AK1, AK2, AK2013}.

In \cite{Alb2} and \cite{AK1}  derivations on various subalgebras
of the algebra $LS(M)$ of locally measurable operators with
respect to  a von Neumann algebra $M$ has been considered. A
complete description of derivations has been obtained in the cases
when $M$ is of type I and III. Derivations on  algebras of
measurable and locally measurable operators, including rather non
trivial commutative case, have been studied  by many authors
\cite{AAK2007, AK2007, AAK2008, AAK2009, Alb2, Ayu1, AK2, AK1,
AK2013, Ber2013, Ber, BPS, Ber2, Ber3}. A comprehensive survey of
results concerning derivations on various algebras of unbounded
operators affiliated with von Neumann algebras can be found in
\cite{AK2}.

If we consider the algebra $S(M)$ of all measurable
operators affiliated with  a type III von Neumann algebra $M$, then it is clear that
$S(M)=M$. Therefore from the results of  \cite{Alb2} it follows that
 for type I$_\infty$ and type III von Neumann algebras $M$ every derivation on $S(M)$ is
automatically inner and, in particular, is continuous in the local
measure topology.
 The problem of description of the structure of
derivations in the case of type II algebras has been open so far
and seems to be rather difficult.

In this connection several open problems concerning innerness and
automatic continuity of derivations on the algebras $S(M)$ and
$LS(M)$ for type II von Neumann algebras have been posed in
\cite{AK2}. First positive results in this direction were recently
obtained in \cite{Ber2},\cite{Ber2013}, where automatic continuity has been
proved for derivations on
 algebras of  $\tau$-measurable and locally
 measurable operators affiliated with properly infinite von Neumann algebras.

 Another problem  in \cite[Problem 3]{AK2} asks the following
question:

Let $M$ be a type II von Neumann algebra with a faithful normal
semi-finite trace $\tau.$ Consider  the algebra $S(M)$  (respectively $LS(M)$) of all
measurable (respectively locally measurable) operators affiliated with $M$ and equipped with the
local measure topology  $t.$ Is every  $t$-continuous derivation
$D:S(M)\rightarrow S(M)$ (respectively, $D:LS(M)\rightarrow LS(M)$) necessarily inner?

The positive answer for the above question have been given independently in
\cite{AK2013} and \cite{Ber2}, \cite{Ber2013}. It should be noted that the proofs of
innerness of continuous derivations on algebras $S(M)$ and $LS(M)$ in
the above mentioned  papers are essentially based on the fact that the
von Neumann algebra $M$ is a subalgebra in the considered algebras.

The present paper is devoted to derivations on the algebra $S_0(M,
\tau)$ of all $\tau$-compact operators affiliated with a
 von Neumann algebra $M$ with a faithful normal semi-finite trace $\tau.$
  Since  the algebra  $S_0(M,\tau)$ does not contain the von Neumann
algebra $M$, one cannot directly apply the methods of the
papers \cite{AK2013}, \cite{Ber2013} in this case.

In Section 2 we give preliminaries from the theory of measurable operators
affiliated with a von Neumann algebra $M$ and, in particular, we recall the notion
of $\tau$-compact operator with respect to $M$.

 In section 3 we
show that if $M$ is  a properly infinite von Neumann algebra with
a faithful normal semi-finite trace $\tau,$ then any derivation
$D:S_0(M, \tau)\rightarrow S_0(M, \tau)$ is automatically $t_\tau$-continuous.

In section 4 we prove if  $M$ is  a von Neumann algebra with a
faithful normal semi-finite trace $\tau,$ then every
$t_\tau$-continuous derivation $D:S_0(M, \tau)\rightarrow S_0(M,
\tau)$ is spatial and implemented by an element from the algebra
$S(M, \tau)$. As a corollary, we obtain that if  $M$ is  a
properly infinite von Neumann algebra then arbitrary derivation
$D:S_0(M, \tau)\rightarrow S_0(M, \tau)$ is spatial.

\section{Algebras of measurable operators}

Let  $B(H)$ be the $\ast$-algebra of all bounded linear operators
on a Hilbert space $H,$ and let $\textbf{1}$ be the identity
operator on $H.$ Consider a von Neumann algebra $M\subset B(H)$
with the operator norm $\|\cdot\|$ and  with a faithful normal
semi-finite trace $\tau.$ Denote by $P(M)=\{p\in M:
p=p^2=p^\ast\}$ the lattice of all projections in $M.$

A linear subspace  $\mathcal{D}$ in  $H$ is said to be
\emph{affiliated} with  $M$ (denoted as  $\mathcal{D}\eta M$), if
$u(\mathcal{D})\subset \mathcal{D}$ for every unitary  $u$ from
the commutant
$$M'=\{y\in B(H):xy=yx, \,\forall x\in M\}$$ of the von Neumann algebra $M.$

A linear operator  $x: \mathcal{D}(x)\rightarrow H,$ where the
domain  $\mathcal{D}(x)$ of $x$ is a linear subspace of $H,$ is
said to be \textit{affiliated} with  $M$ (denoted as  $x\eta M$)
if $\mathcal{D}(x)\eta M$ and $u(x(\xi))=x(u(\xi))$
 for all  $\xi\in
\mathcal{D}(x)$  and for every unitary  $u\in M'.$

A linear subspace $\mathcal{D}$ in $H$ is said to be
\textit{strongly dense} in  $H$ with respect to the von Neumann
algebra  $M,$ if

1) $\mathcal{D}\eta M;$

2) there exists a sequence of projections $\{p_n\}_{n=1}^{\infty}$
in $P(M)$  such that $p_n\uparrow\textbf{1},$ $p_n(H)\subset
\mathcal{D}$ and $p^{\perp}_n=\textbf{1}-p_n$ is finite in  $M$
for all $n\in\mathbb{N}$.

A closed linear operator  $x$ acting in the Hilbert space $H$ is
said to be \textit{measurable} with respect to the von Neumann
algebra  $M,$ if
 $x\eta M$ and $\mathcal{D}(x)$ is strongly dense in  $H.$

 Denote by $S(M)$  the set of all linear operators on $H,$ measurable with
respect to the von Neumann algebra $M.$ If $x\in S(M),$
$\lambda\in\mathbb{C},$ where $\mathbb{C}$  is the field of
complex numbers, then $\lambda x\in S(M)$  and the operator
$x^\ast,$  adjoint to $x,$  is also measurable with respect to $M$
(see \cite{Seg}). Moreover, if $x, y \in S(M),$ then the operators
$x+y$  and $xy$  are defined on dense subspaces and admit closures
that are called, correspondingly, the strong sum and the strong
product of the operators $x$  and $y,$  and are denoted by
$x\stackrel{.}+y$ and $x \ast y.$ It was shown in \cite{Seg} that
$x\stackrel{.}+y$ and $x \ast y$ belong to $S(M)$ and these
algebraic operations make $S(M)$ a $\ast$-algebra with the
identity $\textbf{1}$  over the field $\mathbb{C}.$ Here, $M$ is a
$\ast$-subalgebra of $S(M).$ In what follows, the strong sum and
the strong product of operators $x$ and $y$  will be denoted in
the same way as the usual operations, by $x+y$  and $x y.$

It is clear that if the von Neumann algebra $M$ is finite then every linear operator
affiliated with $M$ is measurable and, in particular, a self-adjoint operator is
measurable with respect to $M$ if and only if all its
 spectral projections belong to $M$.

 Let   $\tau$ be a faithful normal semi-finite trace on
 $M.$ We recall that a closed linear operator
  $x$ is said to be  $\tau$\textit{-measurable} with respect to the von Neumann algebra
   $M,$ if  $x\eta M$ and   $\mathcal{D}(x)$ is
  $\tau$-dense in  $H,$ i.e. $\mathcal{D}(x)\eta M$ and given   $\varepsilon>0$
  there exists a projection   $p\in M$ such that   $p(H)\subset\mathcal{D}(x)$
  and $\tau(p^{\perp})<\varepsilon.$
   Denote by  $S(M,\tau)$ the set of all   $\tau$-measurable operators affiliated with $M.$

    Consider the topology  $t_{\tau}$ of convergence in measure or \textit{measure topology}
    on $S(M, \tau),$ which is defined by
 the following neighborhoods of zero:
$$
U(\varepsilon, \delta)=\{x\in S(M, \tau): \exists\, e\in P(M),
 \tau(e^{\perp})<\delta, xe\in
M,  \|xe\|<\varepsilon\},$$  where $\varepsilon, \delta$ are
positive numbers.

 It is well-known \cite{Nel} that $M$ is $t_\tau$-dense in $S(M, \tau)$
and $S(M, \tau)$ equipped with the measure topology is a complete
metrizable topological $\ast$-algebra.

In the algebra   $S(M, \tau)$ consider the subset  $S_0(M, \tau)$
of all operators $x$ such that given any $\varepsilon>0$ there is
a projection  $p\in P(M)$ with $\tau(p^{\perp})<\infty,\,xp\in M$
and $\|xp\|<\varepsilon.$ Following \cite{Str} let us call the
elements of $S_0(M, \tau)$ \textit{$\tau$-compact operators} with
respect to $M.$ It is known \cite{Mur} that  $S_0(M, \tau)$ is a
$\ast$-subalgebra in  $S(M, \tau)$ and a bimodule over $M,$ i.e.
$ax, xa\in S_0(M, \tau)$ for all $x\in S_0(M, \tau)$ and $a\in M.$

The following properties of the algebra $S_0(M, \tau)$  are known
(see  \cite{Str}):

Let  $M$ be a von Neumann algebra with a faithful normal
semi-finite trace $\tau.$ Then

1) $S(M, \tau)=M+S_0(M, \tau);$

2) $S_0(M, \tau)$ is an ideal in $S(M, \tau).$

Note that if the trace $\tau$ is finite then
$$S_0(M, \tau)=S(M, \tau)=S(M).$$
 It is well-known \cite{Str} $S_0(M, \tau)$ equipped with the measure topology is a complete
metrizable topological $\ast$-algebra.

\section{Continuity of derivations on
$\ast$-algebras of $\tau$-compact operators in the properly infinite case}

\medskip

In this section we prove the automatic continuity of derivations in the measure topology
 on the algebra of  $\tau$-compact operators affiliated with
  a properly infinite von Neumann algebra and a faithful normal
semi-finite trace.

Denote by $S_0(M, \tau)_b$ the bounded part of the $\ast$-algebra
$S_0(M, \tau),$ i.e.
$$
S_0(M, \tau)_b=M\cap S_0(M, \tau).
$$
Since $S_0(M, \tau)$ is a complete metrizable topological
$\ast$-algebra with respect to the measure topology, and the norm
topology on $M$ is stronger than the measure topology, we have
that $S_0(M, \tau)_b$ is a Banach $\ast$-algebra with respect to
the norm topology.

Set
$$
P_\tau(M)=\{p\in P(M): \tau(p)<+\infty\}.
$$
It is clear that
$$
P_\tau(M)=P(M)\cap S_0(M, \tau).
$$

 Let $D$ be a derivation on $S_0(M, \tau).$ Let us
define a mapping $D^\ast:S_0(M, \tau)\rightarrow S_0(M, \tau)$  by
setting
$$
D^\ast(x)=(D(x^\ast))^\ast,\, x\in S_0(M, \tau).
$$
 A direct verification
shows that $D^\ast$  is also a derivation on $S_0(M, \tau).$  A
derivation $D$ on $S_0(M, \tau)$  is said to be

-- \textit{hermitian}, if $D^\ast=D,$

-- \textit{skew-hermitian}, if $D^\ast=-D.$

 Every
derivation $D$ on $S_0(M, \tau)$ can be represented in the form
$D= D_1+ i D_2,$ where
$$
D_1 = (D +D^\ast)/2,\quad D_2  = (D - D^\ast)/2i
$$
 are hermitian derivations on $S_0(M, \tau).$

 It is clear that a derivation $D$   is continuous
  if and only if the hermitian derivations
$D_1$  and $D_2$ are continuous.

Therefore further in this section we may assume that $D$ is a
hermitian derivation.

\begin{lem}\label{norcon}
 Let $M$ be a  von Neumann algebra with a faithful normal
semi-finite trace $\tau$ and let $D:S_0(M, \tau)\rightarrow S_0(M,
\tau)$  be a derivation. Then the following assertions are
equivalent:

i) $D$ is $t_\tau$-continuous;

ii) $D|_{S_0(M, \tau)_b}$ is $\|\cdot\|$-$t_\tau$-continuous.
\end{lem}

\begin{proof} Since the norm topology on $S_0(M, \tau)_b$ is stronger than the
measure topology, it follows that implication $i)\Rightarrow ii)$
is true.

Let us show the converse implication. Note that $(S_0(M, \tau),
t_\tau)$  is an $F$-space. Therefore by the "closed graph theorem"
it is sufficient to show that
the graph of the linear operator $D$ is  closed.

Let $x_n \stackrel{t_\tau}\longrightarrow 0$ and $D(x_n)
\stackrel{t_\tau}\longrightarrow y.$ Take an arbitrary $p\in
P_\tau(M).$ Then $x_n p\stackrel{t_\tau}\longrightarrow 0$ and
$$
D(x_n p)=D(x_n)p +x_nD(p)  \stackrel{t_\tau}\longrightarrow yp.
$$
Therefore there exists a sequence $n_1<n_2<\ldots <n_k<\ldots$
such that $x_{n_k}p\in U\left(\frac{1}{k},
\frac{1}{2^{k+1}}\right).$ For each $k\in \mathbb{N}$ take a
projection $p_k\leq p$ such that $\tau(p-p_k)\leq
\frac{1}{2^{k+1}}$ and $\|x_{n_k}pp_k\|\leq\frac{1}{k}.$ Set
$$
q_i=\bigwedge\limits_{k=i}^\infty p_k, \, i\in \mathbb{N}.
$$
Then
$$
\tau(p-q_i)\leq\sum\limits_{k=i}^\infty \tau(p-p_k)\leq
\frac{1}{2^i}
$$
and
$$
\|x_{n_k}pq_i\|\leq \frac{1}{k}
$$
for all $k\geq i.$ This means that $x_{n_k}pq_i
\stackrel{\|\cdot\|}\longrightarrow 0$ as $k\rightarrow \infty$
for all $i.$ Therefore $D(x_{n_k}pq_i)
\stackrel{t_\tau}\longrightarrow 0.$ Thus
$$
D(x_{n_k}p)q_i=D(x_{n_k}pq_i)-x_{n_k}pD(q_i)
\stackrel{t_\tau}\longrightarrow 0.
$$

On the other hand
$$
D(x_{n_k}p)q_i \stackrel{t_\tau}\longrightarrow ypq_i.
$$
So $ypq_i=0$ for all $i.$ Since $q_i\uparrow p$ we have that
$yp=0.$ Taking into account that $p\in
P_\tau(M)$  is arbitrary and the trace $\tau$ is semi-finite, we obain that $y=0.$
This means that $D$ is $t_\tau$-continuous. The proof is complete.
\end{proof}
\begin{thm}\label{contin}
 Let  $M$ be a properly infinite von Neumann algebra with a faithful normal
semi-finite trace $\tau.$ Then any derivation $D:S_0(M,
\tau)\rightarrow S_0(M, \tau)$  is $t_\tau$-continuous.
\end{thm}

\begin{proof} By Lemma~\ref{norcon} it is suffices show that
$D|_{S_0(M, \tau)_b}$ is $\|\cdot\|$-$t_\tau$-continuous.  Since
$(S_0(M, \tau)_b, \|\cdot\|)$  and $(S_0(M, \tau), t_\tau)$  are
$F$-spaces, as above it is sufficient to show that the graph of the linear
operator  $D|_{S_0(M, \tau)_b}$ is closed. For convenience we
denote $D|_{S_0(M, \tau)_b}$ also by $D$.

Let us suppose the converse, i.e. assume that the graph of $D$ is
not closed. This means that there exists a sequence
$\{x_n\}\subset S_0(M, \tau)_b$  such that $x_n
\stackrel{\|\cdot\|}\longrightarrow  0$ and $D(x_n)
\stackrel{t_\tau}\longrightarrow y\neq 0.$ Taking into account
that $S_0(M, \tau)$ is a  $\ast$-algebra and that $D=D^\ast,$ we
may assume that $y = y^\ast,$ $x_n = x_n^\ast$ for all $n\in
\mathbb{N}.$ In this case, $y = y_+ - y_-,$ where $y_+, y_-\in
S_0(M, \tau)$  are respectively the positive and the negative
parts of $y.$ Without loss of generality, we shall also assume
that $y_+\neq 0,$ otherwise, instead of the sequence $\{x_n\}$  we
may consider the sequence $\{-x_n\}.$ Let us choose $\lambda_0>0$
such that the projection $p = \mathbf{1}- e_{\lambda_0}(y)\neq 0.$
We have that
$$
0 < \lambda_0p \leq  pyp.
$$
Again replacing, if necessary, $x_n$ by $x_n=px_np/\lambda_0,$ we
may assume that
\begin{equation*}\label{tens}
y\geq p,
\end{equation*}
because
$$
D\left(px_np\right)=D(p)x_np+pD(x_n)p+px_nD(p)\rightarrow pyp.
$$

Note that $y\in S_0(M, \tau)$ implies that $p\in P_\tau(M).$  By
the assumption, $M$ is a properly infinite von Neumann algebra and
therefore, there exist pairwise orthogonal projections $p_2,
\ldots , p_k,\ldots$ in $p^\perp M p^\perp$ such that $p_1=p\sim
p_k$
 for all $k\geq 2.$
Let $u_k$ be a partial isometry in $M$ such that $u_k^\ast
u_k=p_1,$ $u_k u_k^\ast=p_k$ for all $k.$ Then
\begin{center}
$ u_k x_n u_k^\ast\rightarrow 0$  as $n\rightarrow\infty$
\end{center}
and
$$
D(u_k x_n u_k^\ast)=D(u_k) x_n u_k^\ast+u_k D(x_n) u_k^\ast +u_k
x_n D(u_k^\ast)\rightarrow u_k y u_k^\ast\geq u_k p u_k^\ast=p_k
$$
for all $k.$

Let $\tau(p)=4\varepsilon.$ Since
$$
u_k x_n u_k^\ast\stackrel{\|\cdot\|}\longrightarrow  0
$$
and
$$
p_kD(u_k x_n u_k^\ast)p_k\stackrel{t_\tau}\longrightarrow  p_ku_k
y u_k^\ast p_k
$$
there exist a projection $q_k\leq p_k$ and a number $n(k)$ such
that
$$
\tau(p_k-q_k)<\varepsilon
$$
and
$$
\|u_k x_{n(k)} u_k^\ast\|<1/k^3, \|(p_kD(u_k x_{n(k)}
u_k^\ast)p_k- p_ku_k y u_k^\ast p_k)q_k\|<1/2.
$$
Taking into account these inequalities and since $D$ is hermitian we
obtain that
\begin{equation}\label{eee}
\tau(q_k)>3\varepsilon
\end{equation}
and
\begin{equation}\label{qqq}
 q_kD(u_k x_{n(k)}u_k^\ast) q_k\geq \frac{1}{2} q_k.
\end{equation}

Now set
$$
x=\sum\limits_{k=1}^\infty k u_k x_{n(k)}u_k^\ast\in S_0(M,
\tau)_b.
$$
Then there exists $\lambda>0$ such that
$$
D(x)\in U(\lambda, \varepsilon).
$$
Therefore
\begin{equation}\label{ggg}
q_kD(x)q_k\in U(\lambda, 2\varepsilon).
\end{equation}

On the other hand,
$$
p_kD(p_kxp_k)p_k=
p_kD(p_k)xp_kp_k+p_kp_kD(x)p_kp_k+p_kp_kxD(p_k)p_k.
$$
Taking into account that
$$
p_kx=xp_k=k p_ku_k x_{n(k)} u_k^\ast p_k=k u_k x_{n(k)} u_k^\ast
$$
and
$$
p_kD(p_k)p_k=0
$$
we obtain
$$
p_kD(x)p_k= p_kD(k u_k x_{n(k)} u_k^\ast )p_k.
$$
Whence
$$
q_kD(x)q_k= q_kD(k u_k x_{n(k)} u_k^\ast )q_k.
$$
By \eqref{qqq} we obtain that
$$
q_kD(x)q_k\geq \frac{k}{2}q_k.
$$
Now using  \eqref{ggg} we have
$$
\frac{k}{2}q_k\in U(\lambda, 2\varepsilon)
$$
for all $k.$ Taking into account \eqref{eee} we have $k/2\leq
\lambda$ for all $k.$ This contradiction shows that $x=0.$ Thus
$D$ is continuous. The proof is complete.
\end{proof}

\section{The Main results}

\medskip

In this section we prove the spatiality  of measure continuous
 derivations on the algebra of $\tau$-compact operators with
respect to a von Neumann algebra and a faithful normal
semi-finite trace.

 Denote by $\mathcal{U}(M)$  and $\mathcal{GN}(M)$ the set of all unitaries in
 $M$ and   the set of all partial isometries  in $M,$ respectively.

 A partial ordering  on the set $\mathcal{GN}(M)$ can be defined
as follows:
$$
u\leq_1 v \Leftrightarrow u u^\ast\leq v v^\ast,\, u= uu^\ast v.
$$
It is clear  that
$$
u\leq_1 v \Leftrightarrow  u^\ast u\leq  v^\ast v,\, u=  v u^\ast
u.
$$

Note that $u^\ast u=r(u)$ is the right support of $u,$ and
$uu^\ast=l(u)$ is the left support of $u.$

Similar to the previous section every derivation $D$ on
$S_0(M, \tau)$ can be represented in the form $D= D_1+ i D_2,$
where
$$
D_1 = (D -D^\ast)/2,\quad D_2  = (D + D^\ast)/2i
$$
 are skew-hermitian derivations on $S_0(M, \tau).$

 It is clear that a derivation $D$   is inner or spatial
  if and only if the both skew-hermitian derivations
$D_1$  and $D_2$ are inner or spatial respectively.

Therefore further in this section we may assume that $D$ is a
skew-hermitian derivation.

Since $S_0(M,\tau)$ is an ideal in $S(M,\tau)$  and algebraic operations on $S(M,\tau)$ are $t_\tau$-continuous, each element $a\in
 S(M, \tau)$  implements a $t_\tau$-continuous derivation on the algebra
  $S_0(M, \tau)$ by the formula
$$
D(x) = ax - xa, \,  x\in S_0(M, \tau).
$$

The main aim of the present paper  is to prove the converse
assertion. Namely, we shall prove the following

\begin{thm}\label{main}
Let $M$ be a    von Neumann algebra with a faithful normal
semi-finite trace $\tau.$ Then every $t_\tau$-continuous
derivation $D:S_0(M, \tau)\rightarrow S_0(M, \tau)$ is spatial and
implemented by an element $a\in S(M, \tau).$
\end{thm}

For the proof of this theorem we need several lemmata.

Set
$$
 \mathcal{GN}_\tau(M)=\{u\in \mathcal{GN}(M): uu^\ast\in
 P_\tau(M)\}.
$$
It is clear that
$$
P_\tau(M)\subset  \mathcal{GN}_\tau(M)\subset S_0(M, \tau)_b.
$$

The following three lemmata have been proved in \cite{AK2013} in
the case of finite von Neumann algebras, but the proofs are similar
in the semi-finite case.

\begin{lem}\label{duu}  For every $v\in \mathcal{GN}_\tau(M)$
the element $vv^\ast D(v)v^\ast$ is hermitian.
\end{lem}

\begin{lem}\label{five}  Let $n\in \mathbb{N}$ be a fixed number and
let  $v\in \mathcal{GN}_\tau(M)$ be  a partially isometry. Then
$$
vv^\ast D(v)v^\ast\geq n vv^\ast
$$
if and only if
$$
v^\ast vD(v^\ast) v\leq -n v^\ast v.
$$
\end{lem}

\begin{lem}\label{orto}  Let   $v_1\in \mathcal{GN}_\tau(M)$ be
a partially isometry  and let  $v_2\in \mathcal{GN}_\tau(pMp),$
where \linebreak  $p=\mathbf{1}- v_1v_1^\ast \vee v_1^\ast v_1
\vee s(iD(v_1v_1^\ast))\vee s(iD(v_1^\ast v_1))$ and $s(x)$
denotes the support of a hermitian element $x.$ Then
$$
 (v_1+v_2)(v_1+v_2)^\ast D(v_1+v_2)(v_1+v_2)^\ast=v_1v_1^\ast D(v_1)v_1^\ast
+v_2v_2^\ast D(v_2)v_2^\ast.
$$
\end{lem}

For each  $n\in \mathbb{N}$ consider the set
$$
\mathcal{F}_n=\{v\in \mathcal{GN}_\tau(M): vv^\ast D(v)v^\ast\geq
n vv^\ast\}.
$$
Note that $0\in \mathcal{F}_n$,  so $\mathcal{F}_n$ is not empty.

\begin{lem}\label{bouo}  Let $\varepsilon_n=\sup\{\tau(uu^\ast): u\in
\mathcal{F}_n\}.$
Then $\varepsilon_n\downarrow 0.$
\end{lem}

\begin{proof}
Since  $\mathcal{F}_n\supset \mathcal{F}_{n+1}$ we have that
$\varepsilon_n\downarrow.$
 Let us show that $\varepsilon_n\downarrow 0.$
 Let us suppose the opposite, e.g.
there exists a number $\varepsilon>0$  such that
$\varepsilon_n\geq 2\varepsilon$ for all $n\in \mathbb{N}.$ There
exists element $v_n \in \mathcal{F}_n$ such that
$$
\tau(v_{n} v_{n}^\ast)\geq \varepsilon
$$
for all $n\geq1.$  Since $v_{n}\in \mathcal{F}_{n}$  we have
\begin{equation}\label{inq}
 v_{n}v_{n}^{\ast} D(v_{n})v_{n}^{\ast} \geq n v_{n}v_{n}^{\ast}
 \end{equation}
for all $n\geq1.$

Now take an arbitrary number $c>0$ and let $n$ be a number such
that $n>c\delta,$ where $\delta=\frac{\textstyle
\varepsilon}{\textstyle 2}.$ Suppose that
$$
v_{n}v_{n}^{\ast} D(v_{n})v_{n}^{\ast}\in
cU\left(\delta,\delta\right)= U\left(c\delta,\delta\right).
$$
Then there exists a projection $p\in M$ such that
\begin{equation}\label{eps}
||v_{n}v_{n}^{\ast}  D(v_{n})v_{n}^{\ast} p||<c\delta,\,
\tau(p^\perp)<\delta.
\end{equation}
Let $v_{n}v_{n}^{\ast}
D(v_{n})v_{n}^{\ast}=\int\limits_{-\infty}^{+\infty}\lambda \, d\,
e_{\lambda}$ be the spectral resolution of $v_{n}v_{n}^{\ast}
D(v_{n})v_{n}^{\ast}.$ From \eqref{eps} using \cite[Lemma
2.2.4]{Mur} we obtain that $e^\perp_{c\delta}\preceq p^\perp.$
Taking into account \eqref{inq} we have that
$v_{n}v_{n}^{\ast}\leq e^\perp_{n}.$ Since $n>c\delta$ it follows
that $e^\perp_{n}\leq e^\perp_{c\delta}.$ So
$$
v_{n}v_{n}^{\ast} \leq e^\perp_{n}\leq e^\perp_{c\delta}\preceq
p^\perp.
$$
Thus
$$
\varepsilon\leq \tau(v_{n}v_{n}^{\ast} )\leq
\tau(p^\perp)<\delta=\frac{\varepsilon}{2}.
$$
This contradiction implies that
$$
v_{n}v_{n}^{\ast} D(v_{n})v_{n}^{\ast}\notin
cU\left(\delta,\delta\right)
$$
for all $n>c\delta.$ Since $c>0$ is arbitrary it follows that the
sequence $\{v_{n}v_{n}^{\ast} D(v_{n})v_{n}^{\ast}\}_{n\geq1}$ is
unbounded in the measure topology. Therefore the set $\{vv^\ast
D(v)v^\ast: v\in \mathcal{GN}_\tau(M)\}$ is also unbounded in the
measure topology.

 On the other hand, the continuity of the derivation $D$
implies that the set $\{xx^\ast D(x)x^\ast: \|x\|\leq 1\}$ is
bounded in the measure topology. In particular,  the set
$\{uu^\ast D(u)u^\ast: u\in \mathcal{GN}_\tau(M)\}$ is also
bounded in the measure topology. This contradiction implies that
$\varepsilon_n\downarrow 0.$ The proof is complete.
\end{proof}

Lemma~\ref{bouo} implies that  there exists a number $k\in
\mathbb{N}$  such that $\varepsilon_n<+\infty$ for all  $n\geq
k.$

\begin{lem}\label{bout}
For $n\geq k$  the set $\mathcal{F}_n$ has a maximal element with respect to the
partial ordering $\leq_1.$
\end{lem}

\begin{proof}
Let $\{v_\alpha\}\subset \mathcal{F}_n$ be a totally ordered net.
Then
$$
v_\alpha v_\alpha^\ast\uparrow p,  \, v_\alpha^\ast
v_\alpha\uparrow q,
$$
where $p, q\in P(M).$ Since $\tau(v_\alpha
v_\alpha^\ast)\leq\varepsilon_n$ for all $\alpha,$ we have
$\tau(p), \tau(q)\leq\varepsilon_n.$ Note that $v_\alpha\in eMe,$
where $e=p\vee q.$ Consider the $L_2$-norm
$$
\|x\|_2=\sqrt{\tau(x^\ast x)},\, x\in eMe.
$$

Let us show that $v_\alpha \stackrel{t_\tau}\longrightarrow v$
for some $v \in \mathcal{F}_n.$ For $\alpha\leq \beta$ we have
\begin{eqnarray*}
\|v_\beta-v_\alpha\|_2 & =& \|l(v_\beta)v_\beta-l(v_\alpha)v_\beta\|_2= \\
& = & \|(l(v_\beta)-l(v_\alpha))v_\beta\|_2\leq
\|l(v_\beta)-l(v_\alpha)\|_2\|v_\beta\|= \\
& = & \sqrt{\tau(l(v_\beta)-l(v_\alpha))}\rightarrow 0,
\end{eqnarray*}
because $\{l(v_\alpha)\}$ is an increasing net of projections.
Thus $\{v_\alpha\}$ is a $\|\cdot\|_2$-fundamental, and hence
there exists an element $v$ in the unit ball $eMe$ such that
$v_\alpha\stackrel{\|\cdot\|_2}\longrightarrow v.$ Therefore
$v_\alpha\stackrel{t_\tau}\longrightarrow v,$ and thus we have
$$
v_\alpha  v_\alpha^\ast \stackrel{t_\tau}\longrightarrow v
v^\ast,\, v_\alpha^\ast v_\alpha\stackrel{t_\tau}\longrightarrow
v^\ast  v.
$$
 Therefore
$$
v v^\ast,  \, v^\ast v\in P_\tau(M).
$$
Thus  $v\in \mathcal{GN}_\tau(M).$

Since   $v_\alpha=v_\alpha v_\alpha^\ast v_\beta$ for all
$\beta\geq \alpha$ we have that $v_\alpha=v_\alpha v_\alpha^\ast
v.$ So $v_\alpha\leq_1 v$ for all $\alpha.$ Since
$v_\alpha\stackrel{t_\tau}\longrightarrow v$ by
${t_\tau}$-continuity of $D$ we have that
$D(v_\alpha)\stackrel{t_\tau}\longrightarrow D(v).$ Taking into
account that $v_\alpha v_\alpha^\ast D(v_\alpha)v_\alpha^\ast\geq
n v_\alpha v_\alpha^\ast$ we obtain $v v^{\ast} D(v)v^{\ast}\geq n
v v^{\ast},$ i.e. $v\in \mathcal{F}_n.$

So, any totally ordered net  in $\mathcal{F}_n$ has the least
upper bound. By Zorn`s Lemma $\mathcal{F}_n$ has a maximal
element, say $v_n.$ The proof is complete.
\end{proof}

 The
following lemma is  one of the key steps in the  proof of the main
result.

\begin{lem}\label{boun}  Let $M$ be a
von Neumann algebra  with a faithful normal semi-finite trace
$\tau.$ There exists a sequence of projections $\{p_n\}$ in $M$
with $\tau(\mathbf{1}-p_n)\rightarrow 0$ such that that
$$
\|vv^\ast D(v)v^\ast\|\leq n
$$
for all $v\in
p_n\mathcal{GN}_\tau(M)p_n=\mathcal{GN}_\tau(p_nMp_n).$
\end{lem}

\begin{proof} Let $v_n$ be a maximal element of $\mathcal{F}_n.$ Put
$$
p_n=\mathbf{1}- v_nv_n^\ast \vee v_n^\ast v_n \vee
s(iD(v_nv_n^\ast))\vee s(iD(v_n^\ast v_n)).
$$

Let us prove that
$$
\|vv^\ast D(v)v^\ast\|\leq n
$$
for all $v\in \mathcal{GN}_\tau(p_nMp_n).$

The case  $p_n=0$  is trivial.

Let us consider the case $p_n\neq 0.$ Take $v\in
\mathcal{GN}_\tau(p_nMp_n).$ Let $v v^\ast
D(v)v^\ast=\int\limits_{-\infty}^{+\infty}\lambda \, d\,
e_{\lambda}$ be the spectral resolution of $v v^\ast D(v)v^\ast.$
 Assume that $p=e_n^\perp\neq 0.$ Then
$$
p v v^\ast D(v)v^\ast p\geq n p.
$$
Denote $u=pv.$ Then since $p\leq  vv^\ast,$  we have
\begin{eqnarray*}
u u^\ast D(u)u^\ast  & =& p v v^\ast p D(pv)v^\ast p= \\
& = & p v v^\ast p D(p)vv^\ast p +p v v^\ast p p D(v)v^\ast p=\\
& = & p v v^\ast p D(p) p vv^\ast  +p v v^\ast D(v)v^\ast p =\\
& = &0+ p v v^\ast D(v)v^\ast p\geq n p,
\end{eqnarray*}
i.e.
$$
uu^\ast  D(u)u^\ast \geq n p.
$$
Since $u u^\ast, u^\ast u\leq p_n=\mathbf{1}- v_nv_n^\ast \vee
v_n^\ast v_n \vee s(iD(v_nv_n^\ast))\vee s(iD(v_n^\ast v_n))$ it
follows that $u$ is orthogonal to $v_n,$ i.e.
$uv_n^{\ast}=v_n^{\ast}u=0.$ Therefore  $w=v_n+u\in
\mathcal{GN}_\tau(M).$ Using  Lemma~\ref{orto} we have
\begin{eqnarray*}
w w^\ast D(w)w^\ast & =& v_n v_n^{\ast}  D(v_n)v_n^{\ast} +u
u^\ast D(u)u^\ast\geq n(v_nv_n^\ast+p)=nww^\ast,
\end{eqnarray*}
because
\begin{eqnarray*}
w w^\ast  & =& (v_n+u)(v_n +u)^{\ast}=v_nv_n^{\ast} +u u^\ast=\\
& =& v_nv_n^\ast+pvv^\ast p=v_nv_n^\ast+p.
\end{eqnarray*}
So
$$
w w^\ast D(w)w^\ast \geq n ww^\ast.
$$
This contradicts with the maximality $v_n^.$ From this
contradiction it follows that $e_n^\perp=0.$ This means that
$$
v v^\ast D(v)v^\ast\leq n v v^\ast
$$
for all $v\in \mathcal{U}_\tau(p_nMp_n).$

Set
$$
\mathcal{S}_n=\{v\in \mathcal{GN}_\tau(M): vv^\ast D(v)v^\ast\leq
-n vv^\ast\}.
$$
By Lemma~\ref{five} it follows that $v\in \mathcal{F}_n$ is a
maximal element of $\mathcal{F}_n$ with respect to the partial ordering
$\leq_1$ if and only if $v^\ast$ is a maximal element of
$\mathcal{S}_n$ with respect to this ordering.

Taking into account this observation  in a similar way we can show
that
$$
v v^\ast D(v)v^\ast\geq -n v v^\ast
$$
for all $v\in \mathcal{GN}_\tau(p_nMp_n).$ So
$$
-n v v^\ast\leq v v^\ast D(v)v^\ast\leq n v v^\ast.
$$
This implies that $v v^\ast D(v)v^\ast\in M$ and
\begin{equation}\label{noreq}
 \|vv^\ast D(v)v^\ast\|\leq n
\end{equation}
for all $v\in \mathcal{GN}_\tau(p_nMp_n).$

Finally let us show that
$$
\tau(\mathbf{1}-p_n) \rightarrow 0.
$$

It is clear that
$$
l(iD(v_nv_n^\ast)v_nv_n^\ast)\preceq v_nv_n^\ast,
$$
$$
r(v_nv_n^\ast iD(v_nv_n^\ast))\preceq v_nv_n^\ast.
$$
Since
$$
D(v_nv_n^\ast)=D(v_nv_n^\ast)v_nv_n^\ast+v_nv_n^\ast
D(v_nv_n^\ast)
$$
we have
\begin{eqnarray*}
\tau(s(iD(v_nv_n^\ast))) & =&
\tau(s(iD(v_nv_n^\ast)v_nv_n^\ast+v_nv_n^\ast
iD(v_nv_n^\ast))\leq \\
& \leq & \tau(s(v_nv_n^\ast)\vee l(iD(v_nv_n^\ast)v_nv_n^\ast)\vee
r(v_nv_n^\ast iD(v_nv_n^\ast)))\leq\\
& \leq & \tau(v_n v_n^\ast)+\tau(v_n v_n^\ast)+\tau(v_n
v_n^\ast)=3\tau(v_n v_n^\ast),
\end{eqnarray*}
i.e.
$$
\tau(s(iD(v_nv_n^\ast))) \leq 3\tau(v_n v_n^\ast).
$$
Similarly
$$
\tau(s(iD(v_n^\ast v_n))) \leq 3\tau(v_n^\ast v_n).
$$
Now taking into account that
$$
v_n v_n^\ast \sim v_n^\ast v_n
$$
we obtain
\begin{eqnarray*}
\tau(\mathbf{1}-p_n) & =& \tau(v_nv_n^\ast \vee v_n^\ast v_n \vee
s(iD(v_nv_n^\ast))\vee s(iD(v_n^\ast v_n)))\leq \\
& \leq & \tau(v_n v_n^\ast)+\tau(v_n^\ast v_n)+3\tau(v_n v_n^\ast)+3\tau(v_n^\ast v_n)=\\
& = & 8\tau(v_n v_n^\ast)\rightarrow 0,
\end{eqnarray*}
i.e.
$$
\tau(\mathbf{1}-p_n) \rightarrow 0.
$$
The proof is complete.
\end{proof}

Let $p\in M$ be a projection. It is clear that the mapping
\begin{equation}\label{reduc}
 D^{(p)}:x\rightarrow pD(x)p,\, x\in pS_0(M, \tau)p
\end{equation}
is a derivation on $pS_0(M, \tau)p=S_0(pMp, \tau_p),$ where
$\tau_p$ is the restriction of $\tau$ on $pMp.$

\begin{lem}\label{sequ}
Let $\{p_n\}$ be the sequence of projections from
Lemma~\ref{boun}. Then for every $n\in \mathbb{N}$ there exists an
element $a_n\in p_nMp_n$ such that
\begin{equation}\label{inner}
D^{(p_n)}(x)=a_n x-xa_n
\end{equation}
for all $x\in S_0(p_nMp_n, \tau_{p_n}).$
\end{lem}

\begin{proof}
Let $n\in \mathbb{N}$ be a fixed number. Since the trace $\tau$ is
semi-finite we have
$$
p_n=\bigvee\{p\in P_\tau(p_nMp_n)\}.
$$
Let us show that the derivation $D^{(p)}$ defined as in
\eqref{reduc},
 maps $pMp$ into itself for all $p\in P_\tau(p_nMp_n).$
 Take $v\in \mathcal{U}(pMp).$ Then $vv^\ast=v^\ast v=p$ and
hence by Lemma~\ref{boun} we have that
\begin{eqnarray*}
D^{(p)}(v)=pD(pvp)p & =& vv^\ast D(v)v^\ast v\in pMp
\end{eqnarray*}
and
$$
\|D^{(p)}(v)\|\leq n.
$$
Since any element from the unit ball of $pMp$ can be represented
of the form
$$
x=\frac{1}{2}(v_1+v_2+v_3+v_4),
$$
where $v_i\in \mathcal{U}(pMp),$ $i\in \overline{1, 4},$  it
follows that
$$
\|D^{(p)}(x)\|\leq 2n.
$$
for all $x\in pMp,$ $\|x\|\leq 1.$ So, for each $p\in
P_\tau(p_nMp_n)$ the derivation $D^{(p)}$
 maps $pMp$ into itself and $\|D^{(p)}\|\leq 2n.$
 By Sakai's
Theorem \cite[Theorem 4.1.6]{Sak1} there is an element $a_p\in
pMp$  such that $\|a_p\|\leq 2n$ and $D^{(p)}(x)=a_p x-xa_p$ for
all $x\in pMp.$

If $\tau(p_n)<+\infty$ setting  $p=p_n,$ we can take  $a_n=a_p.$

Now let us consider the case  $\tau(p_n)=+\infty.$ Since the net
$\{a_p\}_{p\in P_{\tau}(p_nMp_n)}$ is norm bounded we have that it
contains a subnet which *-weakly converges in $p_nMp_n.$ Without
loss of generality we may assume that $a_p\rightarrow a_n$ for
some $a_n\in p_nMp_n.$

Let $x\in M$ be an element with  $s(x)\in P_\tau(p_nMp_n).$ Since
$M$ is semi-finite there exist  $f, e\in P_\tau(p_nMp_n)$ such
that  $s(x)\leq f\leq e.$ Then
$$
fD^{(p_n)}(x)f=fD(x)f=fD^{(e)}(x)f=fa_ex-xa_ef,
$$
i.e.
$$
fD^{(p_n)}(x)f=fa_ex-xa_ef.
$$
Since $\{a_e\}_{e\in J}$ is a subnet of the net $\{a_p\}_{p\in
P_{\tau}(p_nMp_n)},$ where $J=\{e\in P_\tau(p_nMp_n): f\leq e\},$
it follows that $a_e\rightarrow a_n.$ Taking into account that the
left side of the last equality is independent  on $e$ we have
$$
fD^{(p_n)}(x)f=f(a_nx-xa_n)f.
$$
Since $f$ is arbitrary it follows that
$$
D^{(p_n)}(x)=a_nx-xa_n.
$$
Taking into account that the set $\{x\in S_0(p_nMp_n, \tau_{p_n}):
\tau(s(x))<+\infty\}$ is $t_\tau$-dense in $S_0(p_nMp_n,
\tau_{p_n})$ and that $D$ is $t_\tau$-continuous we obtain
$$
D^{(p_n)}(x)=a_n x-xa_n
$$ for all
$x\in S_0(p_nMp_n, \tau_{p_n}).$ The proof is complete.
\end{proof}

\textit{Proof of Theorem~\ref{main}}. Let $\{p_n\}$ be the
sequence of projections from Lemma~\ref{boun} and let $\{a_n\}$ be
the sequence from Lemma~\ref{sequ}. Since
$\tau(\mathbf{1}-p_n)\rightarrow 0,$ there exist a sequence
$n_1<n_2<\ldots <n_k<\ldots $ such that
$\tau(\mathbf{1}-p_{n_k})<1/2^{k+1}.$ Set
$$
q_k=\bigwedge\limits_{i=k}^\infty p_{n_i}
$$
and
$$
b_k=q_k a_{n_{k}}q_k
$$
for all  $k\in \mathbb{N}.$ Then
$$
D^{(q_k)}(x)=b_k x-xb_k
$$
for all $x\in S_0(q_kMq_k, \tau_{q_k}).$

Now we will construct a sequence $\{c_k\}$ such that $c_k\in q_k
Mq_k$ and
$$
D^{(q_k)}(x)=c_k x-xc_k
$$
for all $x\in S_0(q_kMq_k, \tau_{q_k}),$ and moreover
$q_ic_jq_i=c_i$ for all $i<j.$

Set $c_1=b_1$ and suppose that elements $c_1, c_2,\ldots, c_k$ have been
already constructed. Since
$$
D^{(q_k)}(q_kxq_k)=q_kD^{(q_{k+1})}(q_kxq_k)q_k
$$
we obtain that
$$
[c_k, q_kxq_k]=q_k[b_{k+1}, q_kxq_k]q_k=[q_kb_{k+1}q_k, q_kxq_k]
$$
for all $x\in S_0(M, \tau).$ In particular,  the element
$c_k-q_kb_{k+1}q_k$ commutes with any projection from
$P_\tau(q_kMq_k).$  Hence, the element $c_k-q_kb_{k+1}q_k$ belongs to
the center of the algebra $q_kMq_k.$ Then there exists a central
element $f_k$ from $q_{k+1}Mq_{k+1}$ such that
$c_k-q_kb_{k+1}q_k=f_k q_k.$ Set $c_{k+1}=b_{k+1}+f_{k}.$ Then
$$
D^{(q_{k+1})}(x)=c_{k+1} x-xc_{k+1}
$$
for all $x\in S_0(q_{k+1}Mq_{k+1}, \tau_{q_{k+1}}).$ Further
$$
q_kc_{k+1}q_k=q_kb_{k+1}q_k+f_{k}q_k=c_k
$$
and
$$
q_ic_{k+1}q_i=q_iq_kc_{k+1}q_kq_i=q_ic_kq_i=c_i
$$
for $i<k+1.$

Now let us show that  $\{c_k\}$ is a $t_\tau$-fundamental
sequence. Let $\varepsilon>0.$ Take a number $k_\varepsilon$ such
$1/2^{k_\varepsilon}<\varepsilon.$ Put $p=q_{k_\varepsilon}.$  For
$i, j>k_\varepsilon$ we have
$$
p(c_i-c_j)p=q_{k_\varepsilon}c_iq_{k_\varepsilon}-q_{k_\varepsilon}c_jq_{k_\varepsilon}
=c_{k_\varepsilon}-c_{k_\varepsilon}=0.
$$
Since $\tau(\mathbf{1}-q_{k_\varepsilon})<1/2^{k_\varepsilon}$ we
have that $\{c_k\}$ is fundamental in the so-called topology of two-side
convergence in measure. But this topology
 coincides with the topology $t_\tau$ (see \cite[Proposition 3.4.11]{Mur}), and therefore
 there exists an element $c\in S(M, \tau)$ such that $\{c_k\}$
converges to $c$ in $t_\tau$-topology. Since
$$
q_kD(q_kxq_k)q_k=c_k (q_kxq_k)-(q_kxq_k)c_k,
$$
due to the continuity of the derivation $D$ we obtain that
$$
D(x)=cx-xc
$$
for all $x\in S_0(M, \tau).$ The proof is complete.

From Theorems \ref{contin} and \ref{main} we obtain the following
result.

\begin{thm}\label{proper}
 If $M$ is a properly infinite von Neumann algebra with a faithful normal
semi-finite trace $\tau,$ then any derivation $D:S_0(M,
\tau)\rightarrow S_0(M, \tau)$  is  spatial and implemented by an
element $a\in S(M, \tau).$
\end{thm}

 \bigskip

\section*{Acknowledgment}

Part of this work was done within
the framework of the Associateship Scheme of the Abdus Salam
International Centre  for Theoretical Physics (ICTP), Trieste,
Italy. The first author would like to thank ICTP for providing financial support
 and all facilities during his stay at (July-August, 2013).

\end{document}